\documentclass[12pt]{amsart}

\usepackage[english]{babel}
\usepackage[latin1]{inputenc}
\usepackage{amsmath,amssymb, color,enumerate,amssymb,wasysym,mathrsfs}
 \usepackage{mathptmx}
 \DeclareMathAlphabet{\mathcal}{OMS}{cmsy}{m}{n}

\DeclareSymbolFont{rsfscript}{OMS}{rsfs}{m}{n}
\DeclareSymbolFontAlphabet{\mathrsfs}{rsfscript}

\DeclareSymbolFont{AMSb}{U}{msb}{m}{n}
\DeclareSymbolFontAlphabet{\mathbb}{AMSb}

\DeclareSymbolFont{eufrak}{U}{euf}{m}{n}
\DeclareSymbolFontAlphabet{\gothic}{eufrak}

\newcommand{\mb}{\mathbb}

\newcommand{\f}{\mathfrak}

\DeclareMathOperator{\Spec}{Spec}
\newcommand{\rann}{\operatorname{r.ann}}
\newcommand{\lann}{\operatorname{l.ann}}
\newcommand{\Cal}[1]{{\mathcal #1}}

\newtheorem{teor}{Theorem}[section]
\newtheorem{prop}[teor]{Proposition}
\newtheorem{Lemma}[teor]{Lemma}
\newtheorem{cor}[teor]{Corollary}
\newtheorem{Def}[teor]{Definition}
\theoremstyle{remark}

\newtheorem{Ex}[teor]{Example}
\newtheorem{Rem}[teor]{Remark}
\newcommand{\Z}{\mathbb{Z}}
\newcommand{\V}{\mathrm{V}}
\newcommand{\D}{\mathrm{D}}
\begin{document}
  \title[Multiplicative lattices]{Multiplicative lattices: maximal implies prime, and related questions.}
     \author{Alberto Facchini}
\address[Alberto Facchini]{Dipartimento di Matematica ``Tullio Levi-Civita'', Universit\`a di 
Padova, 35121 Padova, Italy}
 \email{facchini@math.unipd.it}
\thanks{The first author is partially supported by Ministero dell'Universit\`a e della Ricerca (Progetto di ricerca di rilevante interesse nazionale ``Categories, Algebras: Ring-Theoretical and Homological Approaches (CARTHA)'') and the Department of Mathematics ``Tullio Levi-Civita'' of the University of Padua (Research programme DOR1828909 ``Anelli e categorie di moduli''). The second author is partially supported by GNSAGA,  the projects PIACERI ``PLGAVA-Propriet\`a locali e globali di anelli e di variet\`a algebriche'' and ``MTTAI - Metodi topologici in teoria degli anelli e loro ideali'' of the University of Catania, and  the research project PRIN ``Squarefree Gr\"obner degenerations, special varieties and related topics''. 
Moreover, both authors are supported by Fondazione Cariverona (Research project ``Reducing complexity in algebra, logic, combinatorics - REDCOM'' within the framework of the programme Ricerca Scientifica di Eccellenza 2018).}
\author{Carmelo Antonio Finocchiaro}
\address[Carmelo Antonio Finocchiano]{Dipartimento di Matematica e Informatica, Universit\`a\ di Catania, Citt\`a\ Universitaria, viale Andrea Doria 6, 95125 Catania, Italy}
\thanks{}
\email{cafinocchiaro@unict.it}


      \begin{abstract} 
      The goal of this paper is to deepen the study of multiplicative lattices in the sense of Facchini, Finocchiaro and Janelidze. We provide a sort of Prime Ideal Principle that guarantees that maximal implies prime in a variety of cases (among them the case of commutative rings with identity). This result is used to study the lattice theoretic counterpart of multiplicative closed sets, that of $m$-systems. The notion of $m$-system is also studied from the topological point of view.    \end{abstract}

    \maketitle


\section{Introduction} 

Multiplicative lattices have been recently introduced in \cite{FFJ}. They provide a natural common framework to present several fundamental notions of Algebra: prime ideals, (Zariski) prime spectra, nilpotency, solvability, central series, derived series, centralizers, center. They also allow to explain why the center of a group is a normal subgroup, and the center of a ring is a subring, but not an ideal. It is important to remark that in our multiplicative lattices multiplication is in general not associative, not commutative, does not have an identity, and no form of distributive law holds. In a multiplicative lattice $L$ in which  multiplication $\cdot$ distributes over join $\vee$, a maximal element $m$ of $L$ is prime if and only if $1\cdot 1\not\leqslant m$.

In this paper we continue our study of multiplicative lattices \cite{F, FFJ, FGT}. We extend the Prime Ideal Principle \cite{LR} from the case of commutative rings to the case of multiplicative lattices (Proposition~\ref{prime-ideal-principle}). We study $m$-systems, giving them a topological characterization (Theorem~\ref{alberto-wow} and Corollary~\ref{compact-saturated}).
We also provide the lattice-theoretic counterpart of the description due to Acosta and Rubio \cite{AR} of Zariski spectra of commutative rings without identity as open subsets of spectral topological spaces (Theorem~\ref{aperti-sono-spettri}). 

\section{Preliminary notions and terminology}

Recall some definitions given in \cite{FFJ}. 
    A {\em multiplicative lattice} $(L, \vee,\cdot)$ is a complete join-semilattice $(L, \vee)$ with a further binary operation $$\cdot{}\colon L\times L\to L$$ {\em (multiplication)} such that $xy\leqslant x$ and $xy\leqslant y$ for all $x,y\in L$.
    
    Clearly, a complete join-semilattice $(L, \vee)$ is a complete lattice, and therefore the conditions $xy\leqslant x$ and $xy\leqslant y$ are equivalent to $xy\leqslant x\wedge y$.
The least element and the largest element of $L$ will be denoted by $0$ and $1$, respectively.

	An element $p\neq1$ in $L$ is {\em prime} if $	xy\leqslant p\Rightarrow(x\leqslant p\,\,\,\text{or}\,\,\,y\leqslant p)$ for every $x,y\in L$. Let $\mathrm{Spec}(L)$ denote the set of all prime elements of $L$.
	
	For any $x\in L$, set
\begin{equation*}
\V_L(x):=\{p\in\mathrm{Spec}(L)\,|\,x\leqslant p\} 
\end{equation*}    and 
\begin{equation*}
\D_L(x):=\Spec(L)\setminus \V(x).
\end{equation*} When there is no danger of confusion, we will simply write $\mathrm{V}(x)$ and $\D(x)$, omitting of indicating the lattice $L$. The sets $\V_L(x)\ (x\in L)$ satisfy the axioms for the closed sets of a topology on $\Spec(L)$. 
This topology is the {\em Zariski topology}, and from now on we assume that $\mathrm{Spec}(L)$ is equipped with this topology, We will call this topological space the \textit{prime spectrum} of $L$. It is always a sober topological space \cite[Lemma 2.6]{FFJ}.

If $f\colon L\to M$ is a morphism of complete join-semilattices, that is, a mapping such that $f(\bigvee X)=\bigvee f(X)$ for every subset $X$ of $L$, then there is a unique morphism $u\colon M\to L$ of complete meet-semilattices such that $f(x)\leqslant y\Leftrightarrow x\leqslant u(y)$  for every $x\in L$ and every $y\in M$. This mapping $u$ is the right adjoint of $f$ when the partially ordered sets $L$ and $M$ are viewed as categories in the usual way. 

The category $\mathsf{CML}$ of complete multiplicative lattices is  the category whose morphisms are  the  morphisms $f\colon L\to M$ of complete join-semilat\-tices  such that  $f(x)f(x')\leqslant f(xx')$ for all $x,x'\in L$. Thus a morphism of complete multiplicative lattice is a mapping $f\colon L\to M$ such that $f(\bigvee X)=\bigvee f(X)$ for every subset $X$ of $L$, $f(x)f(x')\leqslant f(xx')$ for all $x,x'\in L$, and $f(1)=1$. 

\begin{teor} {\rm \cite[Theorem 3.2]{FFJ}} 
	Let $f\colon L\to M$ be a morphism in $\mathsf{CML}$ and let $u\colon M\to L$ be its right adjoint. Then:
	\begin{itemize}
		\item [(a)] If $p$ is a prime element of $M$, then $u(p)$ is a prime element of $L$.
		\item [(b)] The mapping $$\Spec(f)\colon {\Spec}(M)\to\Spec(L)$$ defined by ${\Spec}(f)(p)=u(p)$ of every $p\in\Spec(M)$ is continuous, because $$(\Spec(f))^{-1}(\mathrm{V}(x))=\mathrm{V}(f(x))$$ for every $x\in L$.
		\item [(c)] The assignment $L\mapsto \Spec(L)$ defines a functor $$\Spec\colon\mathsf{CML}^{\mathrm{op}}\to\mathsf{STop},$$ where $\mathsf{STop}$ is the category of sober topological spaces.
		\item [(d)]{\rm \cite[12.3(e)]{FFJ}} The  functor $\Spec\colon\mathsf{CML}^{\mathrm{op}}\to\mathsf{STop}$ is the right adjoint of the functor $\Omega\colon \mathsf{STop}\to \mathsf{CML}^{\mathrm{op}}$, that assigns to each sober space $X$ the complete lattice $\Omega(X)$ of its open subsets with multiplication the intersection of two open subsets.
	\end{itemize}
\end{teor}

Notice that there is not a canonical way of choosing the multiplication in the lattice of all congruences of an algebraic structure. For instance, let $R$ be the commutative ring $\Z_{(p)}$, the DVR obtained localizing the ring $\Z$ of integers at a maximal ideal $(p)$. The lattice $L$ of all ideals of $R$ is isomorphic to the linearly ordered set $\Z_{\le0}\cup\{-\infty\}$. Then:

(1) If we fix on $L$ the usual multiplication of ideals $IJ$, corresponding to the addition in $\Z_{\le0}\cup\{-\infty\}$, then the prime ideals of $R$ are only $0$ and the maximal ideal of $R$ (corresponding to $-1$ and $-\infty$ respectively in the ordered set $\Z_{\le0}\cup\{-\infty\})$. 
     
     (2) If we fix on $L$ the multiplication $IJ+JI$ (which is the Huq=Smith commutator in the protomodular category of commutative rings), we get the same prime ideals as in (1).
     
     (3) If we fix on $L$ the multiplication in which the product is always $0$, we find that there is no prime ideal, i.e., the spectrum of $L$ is empty.
     
     (4) If we fix on $L$ the multiplication in which the product is $I\cap J$, we get that every proper ideal is prime.
     
     (5) If we fix on $L$ the multiplication in which the product is the commutator $[I,J]$ (i.e., the ideal of $R$ generated by all commutators $[i,j]=ij-ji$), then the spectrum is also empty, because $R$ is commutative. It is easy to see that if $R$ is a ring, then  the multiplication on the two-sided ideals of $R$ in which the product of two two-sided ideals $I,J$ is the commutator $[I,J]$ is the multiplication in which the product is always $0$ if and only if the ring $R$ is commutative.

\bigskip

An element $s$ of a multiplicative lattice $L$ is {\em semiprime} if $x^2\leqslant s$  implies $x\leqslant s$ for every $x\in L$. The {\em semiprime radical} of $L$ is the meet of all prime elements of $L$. The multiplicative lattice $L$ is {\em semisimple} if its semiprime radical is $0$. 

	Recall that an element $x$ of a lattice $L$ is {\em meet-irreducible} if $x=y\wedge z$ implies $x=y$ or $x=z$ for every $y,z\in L$. 

\begin{Lemma}\label{3'} {\rm \cite[Lemma 4.7]{F}} Let $L$ be an $m$-distributive algebraic multiplicative lattice. Then an element of $L$ is prime if and only if it is a meet-irreducible semiprime element of $L$.\end{Lemma}

Let $C(L)$ denote the set of all compact elements of $L$. Recall that $L$ is {\em algebraic} if every element of $L$ is the join of some subset of $C(L)$.

 An {\em $m$-system} in $L$ is a nonempty subset $S$ of $C(L)$ such that for every $x,y\in S$ there exists $z\in S$ such that $z\leqslant xy$.  
 
 Let $x$ be an element of a multiplicative lattice $L$. 
 The {\em lower central series} (or {\em descending central series}) of $x$ is the descending series $$x=x_1\ge x_2\ge x_3\ge\dots,$$ where $x_{n+1}:=x_n\cdot x$ for every $n\ge 1$. If $x_n=0$ for some $n\ge 1$, then $x$ is {\em left nilpotent}. The element $x$ is {\em idempotent} if $x_2=x\cdot x=x$, and {\em abelian} if $x_2=x\cdot x=0$. Similarly, the element $x$ is {\em right nilpotent} if $_nx=0$ for some $n$, where now in the descending series  the elements $_nx$ are defined recursively  by $_{n+1}x:=x\cdot {}_nx$. 
 
 The {\em derived series} of $x$ \cite[Definition~6.1]{FFJ}  is the descending series $$x:=x^{(0)}\ge x^{(1)}\ge x^{(2)}\ge\dots,$$ where $x^{(n+1)}:=x^{(n)}\cdot x^{(n)}$ for every $n\ge 0$. The term $x':=x_2=x\cdot x=x^{(1)}$ is the {\em derived element} of $x$. The element $x$ of $L$ is {\em solvable} if $x^{(n)}=0$ for some integer $n\ge0$. 
 
 If the multiplication on the lattice is associative, then left nilpotency, right nilpotency and solvability of an element $x\in L$ coincide.
 
 Notice that the first pages of any standard text of Lie Algebras, like for instance \cite{Lie}, usually don't describe special properties of Lie Algebras, but very standard notions concerning their multiplicative lattice of ideals. Also see \cite{Kawa}.

\section{The monotonicity and $m$-distributivity conditions}

We will say that the  {\em monotonicity condition} holds in a multiplicative lattice $(L,\vee,\cdot)$ if $x\leqslant y$ and $x'\leqslant y'$ imply $xx'\leqslant yy'$ for every $x,y,x',y'\in L$. 

We recall the following facts for the reader's convenience. 

\begin{Lemma}\label{prime-compact}{\rm \cite[Lemma 6.16]{FFJ}}
Let $L$ be a multiplicative lattice satisfying the monotonicity condition, and let $p\in L\setminus\{1\}$. Then $p\in \Spec(L)$ if and only if whenever $a,b\in C(L)$ satisfy $ab\leqslant p$ then either $a\leqslant p$ or $b\leqslant p$. 
\end{Lemma}
	
\begin{Lemma}\label{msystem} {\rm \cite[Lemma 4.2]{F}} Let $(L,\vee,\cdot)$ be an algebraic multiplicative lattice that satisfies the monotonicity condition, and let $C(L)$ denote the set of all compact elements of $L$. An element $p\in L$ is prime if and only if $S_p:=\{\, c\in C(L)\mid c\nleqslant p\,\}$ is an $m$-system in $L$.\end{Lemma}

We will say that  {\em $m$-distributivity} holds in a multiplicative lattice $(L,\vee,\cdot)$ if $$(x\vee y)z=xz\vee yz\quad\mbox{\rm and}\quad x(y\vee z)=xy\vee xz$$ for every $x,y,z\in L$, 

\begin{Lemma}\label{dist-implies-mono}\cite[Lemma 4.3]{F} If $m$-distributivity holds in a multiplicative lattice $L$, then the monotonicity condition holds in $L$.\end{Lemma}

We say that an algebraic multiplicative lattice in which $m$-distributivity holds is {\em hyperabelian} if it satisfies the equivalent conditions of the following theorem:

\begin{teor}\label{hyper}\cite[Theorem 4.11]{F} Let $(L,\vee,\cdot)$ be an algebraic multiplicative lattice in which $m$-distributivity holds. The following conditions are equivalent:

{\rm (a)} $1$ is the unique semiprime element of $L$.

{\rm (b)} For every $x\in L$, $x\ne 1$, there exists $y\in L$ such that $y>x$ and $y^2\leqslant x$.

{\rm (c)} There exists a strictly ascending chain $$0:=x_0< x_1< x_2<\dots< x_{\omega}< x_{\omega+1}< \dots <x_{\alpha}:=1$$ in $L$  indexed in the ordinal numbers less or equal to  $\alpha$ for some ordinal $\alpha$, such that $ x_{\beta+1}^2\leqslant x_{\beta}$ for every ordinal $\beta<\alpha$ and $x_\gamma=\bigvee_{\beta<\gamma}x_\beta$ for every limit ordinal $\gamma\le\alpha$.

{\rm (d)} The lattice $L$ has no prime elements, that is, $\Spec(L)=\emptyset$.

{\rm (e)}  The semiprime radical of $L$ is $1$.

{\rm (f)}  Every $m$-system of $L$ contains $0$.\end{teor}

\begin{prop} Let $L$ be a multiplicative lattice in which $m$-distributivity holds.  An element $m$ maximal in $L$, that is, a maximal element of $L\setminus \{1\}$, is a prime element of $L$ if and only if $1\cdot 1\not\leqslant m$.\end{prop}

\begin{proof} Let $L$ be a multiplicative lattice in which $m$-distributivity holds and $m$ be a maximal element of $L$. We will prove that $m$ is not a prime element of $L$ if and only if $1\cdot 1\leqslant  m$.
	
	If $m$ is not a prime element of $L$, there exist $x,y\in L$ such that $x\not\leqslant     m$, $y\not\leqslant m$ and $xy\leqslant m$. By the maximality of $m$, we get that $x\vee m=1$ and $y\vee m=1$. Thus  $1\cdot 1=(x\vee m)(y\vee m)=xy\vee xm\vee my\vee mm\leqslant xy\vee m\vee m\vee m=m$.
	
	Conversely, suppose $1\cdot 1\leqslant m$. If $m$ is prime, then $1\leqslant m$, so $m=1$, a contradiction. Therefore $m$ is not prime in $L$.\end{proof} 

\begin{prop}\label{prime-associative-case}
Let $L$ a multiplicative lattice in which m-distributivity holds and assume the the multiplication of $L$ is associative. If an element $p\in L\setminus\{1\}$ is not prime, then there exist elements $x,y\in L$ such that $x\nleqslant p,y\nleqslant p, xy\leqslant p$ and $yx\leqslant p$. 
\end{prop}
\begin{proof}
By definition, there are elements $a,b\in L$ such that $a\nleqslant p,b\nleqslant p$ and $ab\leqslant p$. If $ba\leqslant p$ the conclusion follows. Thus we can assume that $ba\nleqslant p$. It follows that $(ba)\vee p\nleqslant p$. Since the multiplication of $L$ is associative, we have $(ba)^2=(ba)(ba)=b(ab)a\leqslant p$. From the m-distributivity we infer that
$$
((ba)\vee p\nleqslant p)^2=(ba)^2\vee (ba)p\vee p(ba)\vee p^2\leq p.
$$
Thus in this case the conclusion follows by taking $x=y=(ba)\vee p$. 
\end{proof}

\section{Oka and Ako subsets: a Prime Ideal Principle}

Let $(L,\vee,\cdot)$ denote a multiplicative lattice. Given elements $a,b\in L$, set
$$
(a:_lb):=\bigvee \{\,x\in L\mid x\cdot b\leq a \,\}\quad\mbox{\rm and}\quad(a:_rb):=\bigvee \{\,x\in L\mid b\cdot x\leq a \,\}.
$$
If $A\subseteq L$ we say that $L$ is \emph{$A$-generated} if every element of $L$ is the supremum of some subset of $A$ (thus, $L$ is algebraic if and only if $L$ is $C(L)$-generated). 
By applying the same argument given in the proof of \cite[Lemma 6.16]{FFJ} we get the following result. 
\begin{Lemma}\label{locally-prime}
Let $L$ be a multiplicative lattice that is $A$-generated, for some subset $A\subseteq L$, and satisfies the monotinicity condition. Then an element $p\in L\setminus\{1\}$ is prime if and only if for every $a,b\in A$,
$
ab\leqslant p $ implies that either $a\leqslant p$ or $b\leqslant p$. 
\end{Lemma}
Again, combining the proof of \cite[Lemma 6.16]{FFJ} and Proposition \ref{prime-associative-case} we get the following fact. 
\begin{Lemma}\label{locally-prime-associative-case}
Let $L$ be a multiplicative lattice with associative multiplication and assume that  $L$ is $A$-generated, for some subset $A\subseteq L$, that satisfies the monotinicity condition. If an element $p\in L\setminus\{1\}$ is not prime, then there exists elements $a,b\in A$ such that $a\nleqslant p,b\nleqslant p, ab\nleqslant p$ and $ba\nleqslant p$. 
\end{Lemma}
\begin{Def} {\rm 
	Let $L$ be a multiplicative lattice and let $A$  be a subset of $L$. 
	
	{\rm (1)} We say that a subset $F$ of $L$ is a \emph{left $A$-Oka subset of $L$} if $1\in F$ and, given elements $a\in A, l\in L$, then  
	$$
	(a\vee l\in F \mbox{ and } (l:_la)\in F) \Rightarrow l\in F.
	$$
	
	{\rm (2)} We say that a subset $F$ of $L$ is a \emph{right $A$-Oka subset of $L$} if $1\in F$ and, given elements $a\in A, l\in L$, then  
	$$
	(a\vee l\in F \mbox{ and } (l:_ra)\in F) \Rightarrow l\in F.
	$$
	
	{\rm (3)} We say that a subset $F$ of $L$ is a \emph{$A$-Oka subset of $L$} if $1\in F$ and, given elements $a\in A, l\in L$, then  
	$$
	(a\vee l\in F, (l:_la)\in F \mbox{ and } (l:_ra)\in F) \Rightarrow l\in F.
	$$
	
	{\rm (4)}  $F$ is \emph{an  $A$-Ako subset of $L$} if $1\in F$ and, given an element $l\in L,a,b\in A$, then 
	$$
	(l\vee a\in F \mbox{ and } l\vee b\in F) \Rightarrow l\vee (a\cdot b)\in F.
	$$}
\end{Def}

\begin{prop}[Prime Ideal Principle]\label{prime-ideal-principle}
	Let $(L,\vee,\cdot)$ be a multiplicative lattice that satisfies the monotonicity condition and that is $A$-generated, for some set $A\subseteq L$. Assume that a subset $F\subseteq L$ satisfies one of the following conditions:

{\rm (1)}  $F$ is a left $A$-Oka subset;
	
	{\rm (2)}  $F$ is a right $A$-Oka subset;
	
	{\rm (3)}  $F$ is an $A$-Oka subset and the multiplication of $L$ is associative; 
	
	{\rm (4)}  $F$ is an $A$-Ako subset.
	
\noindent	Then every maximal element of $L\setminus F$ is prime.  
\end{prop}

\begin{proof}
	Since $p\in L\setminus F$, $p\neq 1$. By contradiction, assume that there exist elements $a,b\in L$ such that $a\cdot b\leqslant p$ and $a\not \leqslant p$, $b\not \leqslant p$. Since $L$ is $A$-generated  and satisfies the monotonicity condition, we can assume, without loss of generality, that $a,b\in A$, in view of Lemma \ref{locally-prime}. 
	
	(1) Assume that $F$ is a left $A$-Oka subset. Then
	 $p\cdot b\leqslant p$ implies that $p\leqslant (p:_lb)$. Moreover, since $a\cdot b\leqslant p$, we have $a\leqslant (p:_lb)$, and as $a\not \leqslant p$ we infer $p<(p:_lb)$. Moreover  $p<p\vee b$, since $b\not \leqslant p$. By the maximality of $p$ in $L\setminus F$, we have $p\vee b\in F$ and $(p:_lb)\in F$ and thus, since $F$ is a left $A$-Oka subset of $L$, it follows that $p\in F$, a contradiction. 
	 
(2) If $F$ is a right $A$-Oka subset a similar argument (to that given in (1)) works. 

(3) Suppose that $F$ is an $A$-Oka subset and that the multiplication of $L$ is associative. Keeping in mind Lemma \ref{locally-prime-associative-case}, we can assume also that $ba\leq p$. Now the same argument of case (1) implies that $p<(p:_lb), p< (p:_r b), p< (p:_r b)$. Since $p$ is maximal in $L\setminus F$, it follows that $(p:_lb),  (p:_r b),  (p:_r b)\in F$ and thus, since $F$ is $A$-Oka, $p\in F$, a contradiction. 
	
(4)	Now assume that $F$ is an Ako subset of $F$. Since $p<a\vee p$ and $p<b\vee p$, we get that $a\vee p\in F$ and $b\vee p\in F$ and thus $ p=p\vee (a\cdot b)\in F$, a contradiction. 
\end{proof}
\begin{Ex} (1)  Let $R$ be a commutative ring with identity. Consider the multiplicative lattice $L$ of ideals of $R$ (where the multiplication is the usual one) and let $A$ be the subset of $L$ consisting of all principal ideals. Clearly $L$ is $A$-generated. In this case (left/right) $A$-Oka and $A$-Ako subsets of $L$ are precisely Oka and Ako families of ideals considered in \cite{LR}. Thus a very special case of Proposition \ref{prime-ideal-principle} is
		\cite[2.4]{LR}. 
		
		(2) Moreover, notice that Proposition \ref{prime-ideal-principle}(3) yields as a particular case \cite[Theorem 2.5]{R-comm-algebra}. 
\end{Ex}

The next proposition is the analogue, for multiplicative lattices, of the fact that an ideal of a commutative ring $R$ maximal with respect to the property of being disjoint from a multiplicatively closed subset of $R$, is a prime ideal of $R$. 

\begin{prop}\label{max}
	Let $L$ be a multiplicative lattice, let $S$ be an $m$-system of $L$  and let 
	$$
	\Sigma:=\{l\in L\mid s\nleqslant l \mbox{ for every } s\in S \}. 
	$$
	
	{\rm (1)}  If $0\notin S$, then the set $\Sigma$ has maximal elements. 
		
			{\rm (2)} If $L$ is $m$-distributive, then $L\setminus \Sigma$ is a $C(L)$-Ako subset of $L$. 
			
			{\rm (3)}  { \rm \cite[Proposition 4.8]{F}} If $L$ is algebraic and $m$-distributive, then every maximal element of $\Sigma$ is prime. 
\end{prop}
\begin{proof}
	(1). Since $0\notin S$, it follows $\Sigma\neq\emptyset$. If $C\subseteq \Sigma$ is a chain and $c:=\bigvee C$, then $c\in \Sigma$ (otherwise, there would exist an element $s\in S$ such that $s\leqslant c$, and since $s$ is compact and $C$ is a chain, we could pick an element $c_0\in C$ such that $s\leqslant c_0$, against the fact that $c_0\in \Sigma$). The conclusion follows from Zorn's Lemma. 
	
	(2). Set $F:=L\setminus\Sigma$. Clearly $1\in F$. Take an element $l\in L$ and compact elements $a,b\in L$ such that $l\vee a$ and $l\vee b\in F$. This means that there exist elements $s,t\in S$ such that $s\leqslant l\vee a,t\leqslant l\vee b$. Since $L$ satisfies the monotonicity condition (Lemma \ref{dist-implies-mono}), we have 
	$$
	st\leqslant (l\vee a)(l\vee b)=l^2\vee lb \vee al\vee ab\leq l\vee ab.
	$$
	Now $S$ is an $m$-system, hence there is an element $u\in  S$ such that $u\leqslant st$, and the conclusion follows. 
	
	(3). It suffices to apply Part (2) and Proposition \ref{prime-ideal-principle}.

\end{proof}

\section{Some intervals and their spectra}

\begin{Rem}\label{intervalli}
Let $L$ be a multiplicative lattice and let $x\leqslant y$ be elements of $L$. Then the interval $[x,y]_L:=\{z\in L\mid x\leq z\leq y \}$ is a multiplicative lattice with respect to the operation $\ast$ on $[x,y]_L$ defined by $z\ast z':=(zz')\vee x$ for every $z,z'\in [x,y]_L$. Notice that in the case $x=0$, the multiplication $\ast$ is the restriction of the multiplication on $L$ to $[x,y]_L$. 
\end{Rem}

We say that a multiplicative lattice is {\em prime} is $0$ is a prime element of $L$, that is, if $x,y\in L$ and  $xy=0$, then either $x=0$ or $y=0$.

\begin{Lemma}\label{chiusi-sono-spettri}
Let $L$ be a multiplicative lattice in which $m$-distributivity holds and let $l \in L$. Consider the multiplicative lattice  $A_l :=[l ,1]_L$, with respect to the multiplication $\ast$ defined in Remark \ref{intervalli}. Then $\Spec(A_l )=\V_L(l )$. In particular, $A_l $ is a prime multiplicative lattice if and only if $l \in \Spec(L)$.
\end{Lemma}
\begin{proof}
Let $p\in \V_L(l )$ (in particular, $p\in A_l $) and let $x,y,\in A_l $ be such that $x\ast y\leqslant p$. Since $xy\leqslant x\ast y$ and $p\in \Spec(L)$, it follows that either $x\leqslant p$ or $y\leqslant p$. Conversely, let $p\in \Spec(A_l )$ and let $a,b\in L$ be such that $ab\leqslant p$. Then the elements $a\vee l ,\ b\vee l \in A_l $ satisfy the inequality 
$$
(a\vee l )\ast (b\vee l )=(ab\vee al  \vee l  b\vee l ^2)\vee l \leqslant p 
$$
and thus, since $p\in \Spec(A_l )$, we infer that $a\vee l \leqslant p$ or $b\vee l  \leqslant p$. It immediately follows that $p\in \V_L(l )$. 

The last statement is now trivial.  
\end{proof}
 An element $x$ of a lattice $L$ is maximal if and only if the interval $[x,1]_L$ has two elements.

Notice that in a multiplicative lattice with exactly two elements $0$ and $1$ one always has $x\cdot 0=0\cdot x=0$ for every $x$. As far as $1\cdot 1$ is concerned, two cases can occur: either $1\cdot 1=1$ or $1\cdot 1=0$. Correspondingly, an element $x$ of a multiplicative lattice $L$ in which $m$-distributivity holds is prime and maximal if and only if the interval $[x,1]$ is  the multiplicative lattice with two elements and with $1\cdot 1=1$. The element $x$ is maximal but not prime  if and only if the interval $[x,1]$ is  the multiplicative lattice with two elements and with $1\cdot 1=0$.

The {\em Jacobson radical} of a multiplicative lattice $L$ is join of the set of all the maximal element of $L$ that are prime elements of $L$.

\section{A topological description of $m$-systems}

We say that an $m$-system $S$ is {\em saturated} if for every $x\in S,\
c\in C(L), \ x\leqslant c$ implies $c\in S$. Every $m$-system $S$ is contained
in a smallest (generates) a saturated $m$-system, consisting of all the
elements $c\in C(L)$ such that $ x\leqslant c$ for some $x\in S$.

\begin{Rem}\label{algebraic-basis}
Let $L$ be an algebraic multiplicative lattice. Then, by the proof of \cite[Theorem 4.4]{FFJ}, the open sets $\D(c)$ of $\Spec(L)$, where $c$ is compact, form a basis for the Zariski topology. 
\end{Rem}
\begin{prop}\label{compact}
	Let $L$ be an algebraic $m$-distributive multiplicative lattice and let $Y\subseteq \Spec(L)$ be compact, with respect to the Zariski topology. Then 
	$$
	S_Y:=\{c\in C(L)\mid c\nleqslant p \mbox{ for all } p\in Y \}
	$$
	is a saturated $m$-system.
\end{prop}
\begin{proof}
	If $S_Y=\emptyset$, by definition $\V(c)\cap Y$ is nonempty for every $c\in C(L)$. Thus the family $\mathcal F:=\{\V(c)\cap Y\mid c\in C(L) \}$, which is clearly closed under finite intersections, is a family of closed subsets of $Y$ with the finite intersection property. Since $Y$ is compact, we can pick a prime element $p_0\in Y\cap \bigcap_{c\in C(L)}\V(c)$. Thus $p_0\geqslant \bigvee C(L)=1$ (this equality holds because $L$ is algebraic), and a fortiori $p_0=1$, a contradiction. It follows that $S_Y\neq \emptyset$. 
	
	Now take elements $x,y\in S_Y$ and assume, by contradiction, that for every compact element $c\in L$ such that $c\leqslant xy$ we have $c\leqslant p$ for some $p\in Y$. In other words, the collection $\mathcal G:=\{Y\cap \V(c)\mid c\in C(L) \mbox{ and  }c\leq xy \}$ of closed subsets of $Y$ has the finite intersection property. Again, the compactness of $Y$ implies that there exists a prime $p_1\in Y$ such that $p_1\geqslant c$ for every compact element $c$ such that $c\leqslant xy$. Since $xy$ is the supremum of a set of compact elements because $L$ is algebraic, we infer that
	$$
	xy=\bigvee\{c\in C(L)\mid c\leqslant xy \}\leqslant p_1.  
	$$
	Since $p_1$ is prime, it follows that either $x\leqslant p_1$ or $y\leqslant p_1$, contradicting that $x,y\in S_Y$. 
	It is now clear that $S_Y$ is saturated. 
\end{proof}

\begin{teor}\label{alberto-wow} Let $L$ be an algebraic $m$-distributive multiplicative 
	lattice, and let $S$ be a subset of $C(L)$. Then the following conditions are equivalent. 

{\rm (1)} $S$ is a saturated $m$-system.

{\rm (2)} The subspace $P(S):=\bigcap_{s\in S}\D(s)$ of $\Spec(L)$ is compact  and $S=S_{P(S)}$. 
\end{teor}

\begin{proof} (1) implies (2).  Let $S$ be a saturated $m$-system and let $P:=P(S)$.   First, suppose $0\in S$. Then $S=C(L)$ and, since $L$ is algebraic, $P=\emptyset$ and obviously $S=S_{\emptyset}$. 
	
	Now assume that $0\notin S$. We first show that 
	$$S=\{\,c\in C(L)\mid c\nleqslant p \mbox{ for all }  p\in P\,\}.$$	
	
	$(\subseteq)$ If $s\in S$, then $s\in C(L)$. Suppose by contradiction
	that $s\le\overline{p}$ for some $\overline{p}\in P$. Then $s\in
	S\cap\{\,c\in C(L)\mid c\leqslant \overline{p}\,\}$, and this contradicts the
	fact that $\overline{p}\in P$.
	
	$(\supseteq)$ Consider the set $T:=\{\,y\in L\mid S\cap\{\, d\in
	C(L)\mid d\leqslant y\,\}=\emptyset\,\}$. Fix an element $c\in C(L)$ such that
	$c\nleqslant p$ for all $p\in P$. Consider the set $$U:=\{\,y\in T\mid
	c\leqslant y\,\}.$$
	
	If $U\ne\emptyset$, it
	is possible to apply Zorn's Lemma to the nonempty set $U$, because if
	$\{\,u_\lambda\mid \lambda\in\Lambda\,\}$ is a nonempty chain in $U$,
	then $u:=\bigvee_{\lambda\in \Lambda }u_\lambda \in L$ is in $U$, because
	if $S\cap\{\, d\in C(L)\mid d\leqslant u\,\}\neq\emptyset$, there exists
	$\overline{d}\in S$ such that $\overline{d}\le
	u=\bigvee_{\lambda\in\Lambda}u_\lambda\in L$. But $\overline{d}$ is
	compact, hence there exists $\lambda_0\in\Lambda$ such that
	$\overline{d}\leqslant u_{\lambda_0}$. Thus $S\cap\{\, d\in C(L)\mid d\le
	u_{\lambda_0}\,\}\ne\emptyset$, which is a contradiction. This shows
	that it is possible to apply Zorn's Lemma to $U$, hence there is a
	maximal element $\overline{p}$ in $U$, which is a maximal element in
	$T$, hence $\overline{p}$ is prime by Proposition~\ref{max}(3). Thus
	$\overline{p}\in\Spec(L)$. Since $\overline{p}\in T$, we know that
	$S\cap\{\, d\in C(L)\mid d\leqslant \overline{p}\,\}=\emptyset$. It follows
	that $\overline{p}\in P$. In particular $c\nleqslant \overline{p}$. But
	$\overline{p}\in U$, so $c\le\overline{p}$, a contradiction. This shows
	that $U=\emptyset$.
	
	From $U=\emptyset$, it follows that $c\notin T$, so $S\cap\{\, d\in
	C(L)\mid d\leqslant c\,\}\ne\emptyset$. Let $s\in S$ be an element in this
	intersection. Then $s\leqslant c$. Since $S$ is saturated, we get that $c\in
	S$, as desired. 

Now we prove that $P$ is compact. Let $\mathcal F:=\{P\cap \V(c_i)\mid i\in I \}$ be a family of closed subsets of $P$ with the finite intersection property. By Remark \ref{algebraic-basis} we can assume that $c_i\in C(L)$ for every $i\in I$. Set
$$
\Sigma:=\{x\in L\mid x\in L\mid c_i\leqslant x \mbox{ and } s\nleqslant x \mbox{ for every }i\in I,\ s\in S \}.
$$
First notice that $\Sigma\neq \emptyset$, because $c:=\bigvee_{i\in I}c_i\in \Sigma$: indeed, if $s\leqslant c$ for some $s\in S$, then $s\leqslant \bigvee_{j\in J}c_j$ for a finite subset $J$ of $I$, since $s$ is compact. Now $\mathcal F$ has the finite intersection property, so there is a prime $p_0\in P\cap \bigcap_{j\in J}\V(c_j)$, and thus $p_0\geqslant \bigvee_{j\in J}c_j\geqslant s$, against the fact that $p_0\in P$. Now let $C\subseteq \Sigma$ be a chain and set $c:=\bigvee C$. If $s\leqslant c$ for some $s\in S$, then $s\leqslant \bigvee H$ for some finite subset $H$ of $C$, and since $H$ is totally ordered we get $s\leqslant h$ for some $h\in H$, against the fact that $H\subseteq \Sigma$. Thus, by Zorn's Lemma, the set $\Sigma$ has a maximal element $p$. We claim that $p$ is prime. In view of Proposition \ref{prime-ideal-principle} it suffices to show that $F:=L\setminus \Sigma$ is a $C(L)$-Ako subset. Since $S\neq \emptyset$, we have $1\in F$. Take elements $l\in L$ and $a,b\in C(L)$ such that $l\vee a,l\vee b\in F$ and assume that $l\vee ab\in \Sigma$, that is, $c_i\leqslant l\vee ab$ and $s\nleqslant l\vee ab$ for every $i\in I$ and $s\in S$. Since $l\vee ab\leqslant( l\vee a)\wedge (l\vee b)$, and $l\vee a,l\vee b\in F$, the only possibility is that $s\leqslant l\vee a$ and $t\leqslant l\vee b$ for some $s,t\in S$. Since $L$ is $m$-distributive (and hence satisfies the monotonicity condition), we have
$$
u\leqslant st\leqslant (l\vee a)(l\vee b)=l^2\vee lb\vee al\vee ab\leqslant l\vee ab
$$
for some $u\in S$, against the assumption that $l\vee ab\in \Sigma$. Thus $F$ is a $C(L)$-Ako subset, $p\in \Spec(L)$ and, by construction, $p$ belongs to the intersection of $\mathcal F$. This proves that $P$ is compact. 

(2) implies (1) follows from Proposition \ref{compact}. \end{proof}

Let $L$ be a multiplicative lattice, and let $\mathcal T$ be the topology on $L$ whose basic open sets are the sets of the type $\V(c)$ for every compact element $c$ of $L$. For every subset $X$ of $\Spec(L)$, let $\overline{X}^{\mathcal T}$ denote the closure of $X$ with respect to $\mathcal T$. 
\begin{Lemma}\label{chiusura-uguaglianza}
Let $L$ be a multiplicative lattice. Then, given subsets $X,Y$ of $\Spec(L)$, $S_X=S_Y$ if and only if $\overline{X}^{\mathcal T}=\overline Y^{\mathcal T}$. 
\end{Lemma}
\begin{proof}
By definition, $S_X=\{c\in C(L)\mid X\subseteq \D(c) \}$. Since every closed subset of $\Spec(L)$, with respect to $\mathcal T$, is of the type $\bigcap_{c \in C}\D(c)$ for some $C\subseteq C(L)$, the conclusion follows immediately. 
\end{proof}
\begin{cor}\label{compact-saturated}
Let $L$ be an algebraic $m$-distributive multiplicative lattice. Then the following properties hold. 

{\rm (1)} If $X\subseteq \Spec(L)$, then $S_X$ is a saturated $m$-system if and only if $\overline{X}^{\mathcal T}$ is compact with respect to the Zariski topology.

	{\rm (2)} Let $\mathcal H(\Spec(L))$ be the family of all the compact saturated subsets of $\Spec(L)$, with respect to the Zariski topology (saturated $:=$ intersection of open sets), and let $\mathcal M(L)$ be the collection of all the saturated $m$-systems of $L$. 
	
	{\rm (a)} $\mathcal H(\Spec(L))$ consists of closed subsets, with respect to the topology~$\mathcal T$. 
		
		{\rm (b)}   The mappings 
		$$
		\varphi: \mathcal H(\Spec(L))\to \mathcal M(L) , \qquad X\mapsto S_X	$$
		$$
		\psi: \mathcal M(L)\to \mathcal H(\Spec(L)), \qquad S\mapsto P(S)
		$$
		are mutually inverse inclusion reversing bijections. 
\end{cor}
\begin{proof}
(1). If $S_X$ is a saturated $m$-system, then 
$$P(S_X)=\bigcap_{s\in S_X}\D(s)=\bigcap_{X\subseteq \D(s),  s\in C(L)}\D(s) =\overline{X}^{\mathcal T}
$$
is compact by Theorem \ref{alberto-wow}. Conversely,  $S_X=S_{\overline{X}^{\mathcal T}}$, by Lemma \ref{chiusura-uguaglianza}, and thus the conclusion follows again  by Theorem \ref{alberto-wow}.

(2). We first prove part (a). Let $K$ be a compact saturated subset of $\Spec(L)$. Thus $K=\bigcap_{i\in I}U_i$, where $U_i$ is open in $\Spec(L)$. Since $L$ is algebraic, each $U_i$ is a union of open sets of the type $\D(c)$, where $c\in C(L)$, and such a union covers the compact space $K$. It follows that for each $i$, there is a finite subset $F_i$ of $C(L)$ such that $K\subseteq \bigcup_{c\in F_i}\D(c)=\D(\bigvee F_i)\subseteq U_i$. A fortiori, $K=\bigcap_{i\in I}\D(\bigvee F_i)$ and since each $\bigvee F_i$ is compact the conclusion follows. 

Part (b) is straightforward. If $X\in \mathcal H(\Spec(L))$, then $\psi\varphi(X)=P(S_X)=\overline{X}^{\mathcal T}=X$ (the last equality follows from part (a)). Finally, if $S\in \mathcal M(L)$, then $\varphi\psi(S)=S_{P(S)}=S$ by Theorem \ref{alberto-wow}. 
\end{proof}

\begin{Rem}\label{spectral-gadgets}
Let $X$ be a spectral space and $\mathcal S$ its set of open subsets. 

(1)  Recall that the \textit{inverse topology} of $X$ is the topology for which a basis of closed sets is the family of open and compact subspaces of $X$ (with respect to $\mathcal S$). The \emph{constructible (or patch) topology} on $X$ is the coarsest topology  on $X$ such that the open compact subsets of $X$ with respect to $\mathcal S$, are the clopen sets in the constructible topology. It is well known that the constructible topology is compact (see, for instance, \cite[Theorem 1]{ho}). By definition, the constructible topology is finer than both $\mathcal S$ and its inverse. It follows that a closed set in the inverse topology of $X$ is compact in the constructible topology (because it is closed in the constructible topology), and therefore a fortiori it is compact with respect to $\mathcal S$.  

(2) Assume that $L$ is a multiplicative lattice and that $\Spec(L)$ is a spectral space for which $\{\D(c)\mid c\in C(L) \}$ is a basis of open compact subsets (this happens whenever the assumptions of \cite[Theorem 4.4]{FFJ} are satisfied). Then, by definition, the topology $\mathcal T$ of $\Spec(L)$ given in Lemma \ref{chiusura-uguaglianza} is exactly the inverse of the Zariski topology. Moreover, the set of closed subsets of $\Spec(L)$ in the inverse topology is  $\mathcal H(\Spec(L))$, in view of \cite[Lemma 2.1]{FFS-smithpowerdomain}. Furthermore, if $\mathcal H(\Spec(L))$ is endowed with the upper Vietoris topology (that is, the topology for which a subbasis of open sets is given by the sets $\{K\in \mathcal H(\Spec(L))\mid K\subseteq \Omega \}$, where $\Omega$ runs in the family of compact open subsets of $\Spec(L)$) and $\mathcal M(L)$ is equipped with the topology generated by the sets of the type $\{S\in \mathcal M(L)\mid c\in S \}$ (where $c\in C(L)$), then it is immediately seen that the mapping $\varphi$ is a homeomorphism. In particular, $\mathcal M(L)$ is a spectral space, because so is $\mathcal H(\Spec(L))$ by \cite[Proposition 3.1 and Theorem 3.4]{FFS-smithpowerdomain}. 
\end{Rem}

\section{Direct product of multiplicative lattices}

Let $L_1,L_2$ be two multiplicative lattices, and consider their direct product $L_1\times L_2$, in which the order is defined componentwise, and the product is $(\ell_1,\ell_2)(\ell'_1,\ell'_2)=(\ell_1\ell'_1,\ell_2\ell'_2)$. Then $(1,0)(0,1)=(0,0)\leqslant p$ for every prime element $p$ of $L_1\times L_2$. Hence either $p=(p_1,1)$ or $p=(1,p_2)$. The following result follows easily:

\begin{prop} Let $L_1,L_2$ be two multiplicative lattices. Then $\Spec(L_1\times L_2)$ is the disjoint union of two clopen subsets homeomorphic to $\Spec(L_1)$ and $\Spec(L_2)$ respectively.\end{prop}

\begin{prop} Let $n_1,n_2$ be two elements of a multiplicative lattice $L$. Then:

{\rm (a)} $\V(n_1)\cap \V(n_2)=\emptyset$ if and only if the interval $[n_1\vee n_2,1]_L$ is a hyperabelian multiplicative lattice.

{\rm (b)} $\V(n_1)\cup \V(n_2)=\Spec(L)$ if and only if the $n_1\cdot n_2\le\sqrt{0}$, where $\sqrt{0}:=\bigwedge\Spec(L)$ denotes the radical of the multiplicative lattice $L$.\end{prop}

\begin{proof} (a) $\V(n_1)\cap \V(n_2)=\emptyset$ if and only if $\V(n_1\vee n_2)=\emptyset$, if and only if there is no prime element $p$ of $L$ such that $p\geqslant n_1 $ and $p\geqslant n_2$, that is, if and only if the interval $[n_1\vee n_2,1]_L$ is hyperabelian, by Lemma \ref{chiusi-sono-spettri}.

(b) $\V(n_1)\cup \V(n_2)=\Spec(L)$ if and only if $\V(n_1\cdot n_2)=\Spec(L)$, that is, if and only if the $n_1\cdot n_2\leqslant\sqrt{0}$.
\end{proof}

As an immediate corollary we get:

\begin{cor} Let $n_1,n_2$ be two elements of a multiplicative lattice $L$. Then $\Spec(L)$ is the disjoint union of two clopen subsets $\V(n_1)$ and $\V(n_2)$ if and only if $n_1\cdot n_2\le\sqrt{0}$ and $[n_1\vee n_2,1]_L$ is hyperabelian.\end{cor}

We are ready to define $n$-systems, similarly to the definition of $n$-systems in rings given by Lam \cite[p.~167]{Lam}.  An {\em $n$-system} in $L$ is a nonempty subset $S$ of $C(L)$ such that for every $x\in S$ there exists $y\in S$ such that $y\leqslant x^2$. 

\begin{Lemma} Let $(L,\vee,\cdot)$ be an algebraic multiplicative lattice that satisfies the monotonicity condition, and let $C(L)$ denote the set of all compact elements of $L$. An element $s\in L$ is semiprime if and only if $S_s:=\{\, c\in C(L)\mid c\nleqslant s\,\}$ is a $n$-system in $L$.\end{Lemma}

\begin{proof} If $s$ is semiprime and $c\in S_s$, then $c\nleqslant s$, so that $c^2\nleqslant s$. Since $L$ is algebraic, there exists $y\in C(L)$ such that $y\leqslant c^2$ and $y\nleqslant s$. Thus $y\in S_s$ and $S_s$ is an $n$-system.

Conversely, if $s$ is not semiprime, there exists $t\in L$ such that $t^2\leqslant s$ but $t\nleqslant s$. Since $t$ is the join of the compact elements $c\leqslant t$, there exists a compact element $c\leqslant t$ with $c\not\leqslant s$. Then $c^2\leqslant t^2\leqslant s$, but $c\leqslant t$ with $c\not\leqslant s$. Hence $c\in S_s$, so that $S_s$ is not an $n$-system.\end{proof}

Notice that in \cite[Section 4]{FGT} there is an example of an $n$-system of a group that cannot be obtained as a
union of $m$-systems, a fact which is in contrast to the behaviour of $n$-systems in
rings.

\section{A case when open subsets of $\Spec(L)$ are spectra}
Let $L$ be a multiplicative lattice and let $x\in L$. The \emph{right annihilator} (resp. \emph{left annihilator}) of $x$ is the element $\rann_L(x):=(0:_rx)=\bigvee\{y\in L\mid xy=0 \}$ (resp., $\lann_L(x):=(0:_lx)=\bigvee\{y\in L\mid yx=0\}$). 

\begin{Rem}
We say that a multiplicative lattice $L$ has \emph{infinite $m$-distributivity} if $\left(\bigvee X\right)\left(\bigvee Y\right)=\bigvee_{(x,y)\in X\times Y}x y$ for every subset $X,Y$ of $L$. For such a multiplicative lattice $L$ one has $x\rann_L(x)=\lann_L(x)x=0$ for every element $x\in L$. 
\end{Rem}
 
 For instance, if $L:=\Cal N(G)$ is the multiplicative lattice of all normal subgroups of a group $G$, then, for every $N\in\Cal N(G)$, $\rann_L(N)=\lann_L(N)$ is the centralizer of $N$ in $L$. It is also possible to define the {\em right center} $rZ(x):=x\wedge \rann_L(x)$ and the
{\em left center} $lZ(x):=x\wedge \lann_L(x)$ of any element $x\in L$. 

Let us recall the following result from \cite[Lemma~4.13]{F}:

\begin{Lemma}\label{4.13} Let $L$ be a multiplicative lattice with infinite $m$-distributivity, $h\in L$, and $n\in [0,h]_L$. If $0$ is a prime element of $[0,n]_L$, then $\lann_{[0,h]}(n)$ coincides with $\rann_{[0,h]}(n)$ and is the unique element $p\in\Spec([0,h]_L)$ such that $p\wedge n=0$.
\end{Lemma}

A consequence of the previous lemma is the following lying-over property for multiplicative lattices. 
\begin{prop}\label{lying-over}
Let $L$ be a multiplicative lattice with infinite $m$-distribu\-tiv\-ity and let $n\in L$. Then for every $q\in \Spec([0,n]_L)$ there exists a unique $p\in \Spec(L)$ such that $p\wedge n=q$. 
\end{prop}
\begin{proof}
Take a prime $q\in \Spec([0,n]_L)$ and consider the multiplicative lattices $M:=[q,n]_L\subseteq M':=[q,1]_L$, with respect to the multiplication $\ast$ defined in Remark \ref{intervalli}.  In view of the last statement of Lemma \ref{chiusi-sono-spettri}, $M$ is a prime multiplicative lattice and $q$ is its zero element. By Lemma \ref{4.13}, there exists a unique prime $p\in \Spec(M')$ such that $p\wedge n=q$. Again by Lemma \ref{chiusi-sono-spettri}, $\Spec(M')=\V_L(q)$ and thus $p$ is prime in $L$, in particular. Moreover, a prime $p'$ of $L$ such that $p'\wedge n=q$ satisfies $q\leqslant p'$ and thus $p'\in \Spec(M')$. The uniqueness of such a prime $p$ in $\Spec(L)$ is clear. 
\end{proof}

We can now state and prove the main result of this section. 

\begin{teor}\label{aperti-sono-spettri}
Let $L$ be a multiplicative lattice with infinite $m$-distributivity, let $n\in L$ and consider the open subspace $\D_L(n):=\Spec(L)\setminus \V_L(n)$ of $\Spec(L)$. Then the mapping $\varphi: \D_L(n)\to \Spec([0,n]_L)$, $p\mapsto p\wedge n$, is a homeomorphism. 
\end{teor}
\begin{proof}
Let $p\in \D_L(n)$. In particular, $p\wedge n<n$. If $a,b\in M:=[0,n]_L$ satisfy $ab\leqslant p\wedge n$ (and thus $ab\leqslant p$), then either $a\leqslant p$ or $b\leqslant p$, because $p$ is prime in $L$. Then either $a\leqslant p\wedge n$ or $b\leqslant p\wedge n$, because $a,b\in M$. Therefore $\varphi$ is well defined. From Proposition \ref{lying-over} it immediately follows that $\varphi$ is bijective. By definition, an open subset of $\D_L(n)$ is of the type $\D_L(l n)$, where $l\in L$ is arbitrary. A straightforward verification shows that $$(\star)\qquad \varphi(\D_L(l n))=\D_M(l n),$$ and thus $\varphi$ is open. By definition, an open subset of $\Spec(M)$ is of the type $\D_M(m)$, for some $m \in M$. Since $m\leqslant n$ it follows $\D_M(m)=\D_M(m n)=\varphi(D_L(m n))$, and thus $(\star)$ and the fact that $\varphi$ is bijective imply  $$\varphi^{-1}(\D_M(m))=D_L(m n).$$ This proves that $\varphi$ is continuous. 
\end{proof}
Notice that the previous theorem is the lattice-theoretic counterpart of the following beautiful result by Acosta and Rubio (see \cite[Theorem 2]{AR}). For a ring $A$ let $\Spec(A)$ denote the set of prime ideals of $A$. 

\begin{cor}\label{acosta-rubio-general}
Let $S\subseteq R$ be commutative rings (not necessarily with identity), and assume that $S$ is an ideal of $R$. Then $\Spec(S)$ is homeomorphic to the open set $\Spec(R)\setminus \V(S)$ of $\Spec(R)$. More precisely, the mapping $\Spec(R)\setminus \V(S)\to \Spec(S)$, $\f p\mapsto \f p\cap S$, is a homeomorphism. 
\end{cor}

\begin{Ex}
Let $V$ be a valuation domain whose maximal ideal $\f m$ is the union of the prime ideals of $V$ that are strictly contained in $\f m$ (such a ring exists: indeed, it is sufficient to take a valuation domain $V$ whose set of non-maximal prime ideals is order isomorphic to $(\mb N,\leqslant)$ \cite[Section 18, Exercise 9]{g}). The open subspace $X:=\Spec(V)\setminus\{\f m\}$ is not compact, because the collection of closed subsets $\{X\cap \V(xV)\mid x\in \f m \}$ of $X$ has the finite intersection property (since $\f m=\bigcup\{\f p\mid \f p\in X \}$) and empty intersection. Now let $L$ be the multiplicative lattice of all the ideals of $V$ and let $M:=\{\mbox{proper ideals of }V\}=[0,\f m]_L$. By  Theorem \ref{aperti-sono-spettri}, $\Spec(M)=X$ and it is not compact. Moreover, if $\f a$ is a compact element of $M$,  it is immediately seen that it is  a finitely generated ideal of $V$ and thus $\f a$ is principal, since $V$ is in particular a Bézout domain. Thus $\f a=xV$, for some $x\in \f a\subseteq \f m$. By construction, there exists a prime ideal $\f p\in X$ such that $x\in \f p$. It follows that $$S_X:=\{\f a\in C(M)\mid \f a\nsubseteq \f q, \mbox{ for every }\f q\in X\}=\emptyset$$  is not an m-system in $M$. 
\end{Ex}
\end{document}